\documentclass[12pt]{amsart}
\usepackage{amssymb}
\usepackage{amsmath}
\usepackage{mathtools}
\usepackage{amsthm}
\newtheorem{theorem}{Theorem}[section]
\newtheorem{lemma}[theorem]{Lemma}
\newtheorem{corollary}[theorem]{Corollary}
\newtheorem{proposition}[theorem]{Proposition}
\theoremstyle{definition}

\newtheorem{remark}{Remark}

\numberwithin{equation}{section}
\usepackage{hyperref}
\hypersetup{
    colorlinks=true, 
    linkcolor=blue,  
}
\mathchardef\hyphen="2D

\makeatletter
\newcommand{\proofpart}[2]{%
    \par
    \addvspace{\medskipamount}%
    \noindent\emph{#2:}\par\nobreak
    \addvspace{\smallskipamount}%
    \@afterheading
}
\makeatother

\begin{document}
\allowdisplaybreaks
\title[Convergence rate of empirical measures in the SRW distance]{Convergence rate of empirical measures in the subspace robust wasserstein distance}
\author{Dakshesh Vasan}
\address{Department of Statistics and Probability, Michigan State University, East Lansing, MI 48825}
\email{dakshesh@msu.edu}
\begin{abstract}
We obtain an estimate for the expected subspace robust 2-Wasserstein distance between any probability measure with a finite $q$-th moment ($q > 2)$ defined on a separable Hilbert space, and its empirical distribution from $n$ i.i.d. samples.
\end{abstract}
\keywords{subspace robust Wasserstein distance; empirical measure; convergence rate}
\subjclass[2020]{60B10, 60B11, 62R20}
\maketitle
\newcommand{\summ}[2]{\underset{i=#1}{\overset{#2}{\sum}}}

\section{Introduction}
\subsection{Background}\label{bg}
Studies involving real data typically depend on empirical distributions, which makes it crucial to find the number of samples required to approximate the true distribution within an allowed error. More precisely, let $\mu$ be a probability measure on a separable Hilbert space $\mathcal{H}$, and let $X_1, X_2,\ldots, X_n$ be independent random vectors in $\mathcal{H}$, identically distributed according to $\mu$. Let $\delta_x$ denote the probability measure on $\mathcal{H}$ with an atom of mass 1 at $x \in \mathcal{H}$. We are interested in the rate of convergence of the empirical distribution $\mu_n = \frac{1}{n}\summ{1}{n}\delta_{X_i}$ to $\mu$ as $n \rightarrow \infty$ in some distance. One such distance is the 2-Wasserstein distance, which has garnered significant interest in optimal transport (OT) theory and in fields such as statistical learning and high-dimensional statistics (see the prefaces of \cite{Chewi} and \cite{Panaretos}).   

The classical \textbf{2-Wasserstein distance} between two probability measures $\mu, \nu$ on a separable Hilbert space $\mathcal{H}$ is defined as follows:

\begin{equation}\label{wassersteindefn}
    W_2(\mu, \nu) \coloneq \inf_{\pi \in \Pi(\mu, \nu)} \left( \int \left\lVert x-y \right\rVert^2 d\pi(x,y) \right)^{1/2}, 
\end{equation}
where $\Pi(\mu, \nu)$ is the collection of all couplings of $\mu$ and $\nu$ (see the precise definition in Section \ref{coupling} below).

For the classical Wasserstein distance on $\mathbb{R}^d$, the number of samples required to obtain a fixed accuracy grows exponentially with the dimension \cite{Dudley, Fournier}. To address this curse of dimensionality (see \cite[Introduction]{Lin2020} for a commentary), several variants of the Wasserstein distance have been proposed. One such variant proposed in existing literature is the projection robust 2-Wasserstein (PRW) distance:

If $\mu, \nu$ are probability measures on $\mathcal{H}$, then for $1 \leq k \leq \mathrm{dim}(\mathcal{H})$, the \textbf{$k$-dimensional projection robust 2-Wasserstein distance} between $\mu$ and $\nu$ \cite[Definition 1]{Paty} is:  
    \begin{equation*}\label{PRW}
    PRW_k(\mu,\nu) := \sup_{E\in\mathcal{G}_k(\mathcal{H})}\inf_{\pi \in \Pi(\mu, \nu)}  \left(\int \left\lVert P_E(x-y) \right\rVert^2 d\pi(x,y) \right)^{1/2} = \sup_{E \in \mathcal{G}_k(\mathcal{H})}W_2(P_{E_{\#}}\mu, P_{E_{\#}}\nu),
    \end{equation*} 
where $\mathcal{G}_k(\mathcal{H})$ is the collection of all $k$-dimensional subspaces of $\mathcal{H}$, $P_E$ is the orthogonal projection from $\mathcal{H}$ onto $E$, $\Pi(\mu, \nu)$ is the collection of all couplings of $\mu$ and $\nu$, and $P_{E_{\#}}\mu, P_{E_{\#}}\nu$ are the pushforwards of $\mu, \nu$ respectively (see Section \ref{coupling} for the precise definition). The empirical convergence rate in the PRW distance has been studied in several works such as \cite{March, Lin, Nietert, NR}. Polynomial rates in $n$ with explicit dependence on $d$ have been obtained on $\mathbb{R}^d$ under assumptions of transport inequalities \cite[Proposition 8]{NR}, Bernstein-type tail or Poincar\'e conditions \cite[Theorem 3.5 and Theorem 3.6]{Lin}. For the $k=1$ case, polynomial rates with no explicit dependence on $d$ have been obtained under log-concavity assumptions in \cite[Theorem 1]{Nietert} and for compactly supported measures in \cite[Corollary 1.3]{March}. While these rates are better than those for the classical 2-Wasserstein distance, the PRW distance involves a non-convex max-min problem. To the best of our knowledge, existing algorithms guarantee only approximate stationary points (see for example \cite{Jiang, Lin2020}). Paty and Cuturi \cite{Paty} introduced a min-max reformulation of the PRW distance, the subspace robust 2-Wasserstein (SRW) distance, which admits a tight convex relaxation making its computation tractable. It is defined as follows.

If $\mu, \nu$ are probability measures on $\mathcal{H}$, then for $1 \leq k \leq \mathrm{dim}(\mathcal{H})$, the \textbf{$k$-dimensional subspace robust 2-Wasserstein distance} between $\mu$ and $\nu$ \cite[Definition 2]{Paty} is:  
    \begin{equation}\label{SRW}
    S_k(\mu,\nu) := \inf_{\pi \in \Pi(\mu, \nu)} \sup_{E\in\mathcal{G}_k(\mathcal{H})} \left(\int \left\lVert P_E(x-y) \right\rVert^2 d\pi(x,y) \right)^{1/2},
    \end{equation} 
with notation as above. 
\begin{remark}
    $S_k(\cdot, \cdot)$ is a metric on the set of all probability measures $\mu$ on $\mathcal{H}$ such that $\int_{\mathcal{H}} \|x\|^2 \, d\mu(x) < \infty$. This follows from standard arguments, and is proved for $\mathcal{H}=\mathbb{R}^d$ in \cite[Proposition 1]{Paty}.
\end{remark}
\begin{remark}\label{S_1}
    For $k = 1$, \[S_1(\mu, \nu) = \inf_{\pi \in \Pi(\mu, \nu)} \sup_{\|w\| \leq 1} \left( \int \langle w, x-y\rangle^2\text{ } d\pi(x,y) \right)^{1/2}.\]
\end{remark}
\begin{remark}[\cite{Paty}, Proposition 2]\label{SRWAlternateForms} 
    On $\mathbb{R}^d,$ the SRW and Wasserstein distances are equivalent. Let $\mu, \nu$ be two probability measures with finite second moments, defined on $\mathbb{R}^d$. Then, for any $1 \leq k \leq d$, 
    \begin{equation*}
         \sqrt{\frac{k}{d}}\cdot W_2(\mu, \nu) \leq S_k(\mu, \nu) \leq W_2(\mu, \nu).
    \end{equation*}
\end{remark}
Experiments by Paty and Cuturi show that the SRW distance is robust to random perturbation of the data \cite[Section 6.3 and Section 7]{Paty}. Lin et al. found from their experiments that both the SRW and PRW distances are more robust to noise than the Wasserstein distance, with the PRW distance being more robust under moderate to high variance, and noted that the SRW distance has more discriminative power than the PRW distance \cite[Section 4]{Lin2020}. They also noted that the equivalence captured in Remark \ref{SRWAlternateForms} suggests that the SRW distance suffers from the curse of dimensionality in theory, whereas Dhouib et al. argue that the SRW distance helps tackle the curse of dimensionality for high-dimensional data \cite[Section 2]{Dhouib}. However, the convergence rate of empirical measures in the SRW distance has not yet been established.

\subsection{Contribution}
In this paper, we establish a dimension-free empirical rate of convergence in the SRW distance. We show that for any probability measure $\mu$ on a separable Hilbert space with bounded $q$-th moment for some $q > 2$, the expected $k$-dimensional SRW distance between $\mu$ and its empirical measure from $n$ independent samples is at most $C_q \sqrt{\min\{k, \log n\} \log \log n/\log n}$, where $C_q > 0$ depends only on $q$. We also prove that, in the worst case, the expected $k$-dimensional SRW distance is at least of order $\sqrt{\min\{k, \log n\}/{\log n}}$, so that the upper bound is sharp up to a $\sqrt{\log \log n}$ factor. Finally, we show that faster rates hold under stronger spectral decay assumptions on $\int_{\mathcal{H}}(x\otimes x)\, d\mu(x)$; for example, under polynomial decay of its eigenvalues. 

\subsection{Main result} 

\begin{theorem}\label{Mainthm}
Suppose that $\mathcal{H}$ is a separable infinite-dimensional Hilbert space. Let $k \in \mathbb{N}, q > 2$ and \[ \mathcal{P}_q :=\left\{ \mu\in\mathcal{P}(\mathcal{H}) :\int_{\mathcal{H}}\|x\|^q\,d\mu(x)\leq 1 \right\},\] where \(\mathcal{P}(\mathcal{H})\) denotes the collection of all probability measures on \(\mathcal{H}\). Given any $\mu \in \mathcal{P}_q$ and $n \geq 3$, let $\mu_n = \frac{1}{n}\summ{1}{n} \delta_{X_i}$ be the empirical distribution of $\mu$ from $n$ i.i.d. samples $X_1, \ldots, X_n$ of $\mu$. Then, 
\begin{equation*}
  c\sqrt{\frac{{\min \{k, \log n\}}}{{\log n}}} \leq \underset{\mu\in\mathcal{P}_q}{\sup}\left(\mathbb{E}S_k^2(\mu, \mu_n)\right)^{1/2}\leq C_q\sqrt{\frac{ {\min\{k, \log n\}\cdot \log \log n}}{{\log n}}},  
\end{equation*}
where $c > 0$ is a universal constant and $C_{q}>0$ depends only on $q$.  
\end{theorem}
\begin{corollary}\label{finim}
    The upper bound in Theorem \ref{Mainthm} also holds in the finite dimensional setting. Let $d \in \mathbb{N}$, $q > 2$, $1 \leq k \leq d$, and let $\mathcal{P}_q^{(d)} :=\left\{
\mu\in\mathcal{P}(\mathbb{R}^d)
:\int_{\mathbb{R}^d}\|x\|^q\,d\mu(x)\leq 1
\right\}$, where $\mathcal{P}(\mathbb{R}^d)$ is the collection of all probability measures on $\mathbb{R}^d$. Then, given $n \geq 3$, we have\[\sup_{\mu\in\mathcal{P}_q^{(d)}}\left(\mathbb{E}S_k^2(\mu, \mu_n)\right)^{1/2}\leq C_q\sqrt{\frac{ {\min\{k, \log n\}\cdot \log \log n}}{{\log n}}}.\] 
\end{corollary}
\begin{proof}
    For $\mu \in \mathcal{P}_q^{(d)}$ and a linear isometry $T$ from $\mathbb{R}^d$ into $\mathcal{H}$, we have $T_{\#}\mu \in \mathcal{P}_q$. So, the claim follows by Proposition \ref{prop} below, and Theorem \ref{Mainthm}.
\end{proof}
\begin{remark}
    Note that in the finite dimensional setting of Corollary \ref{finim}, $\left(\mathbb{E}S_k^2(\mu, \mu_n)\right)^{1/2}$, we also have a bound of order $n^{\frac{1}{4 \vee d}\wedge \left(\frac{1}{2}-\frac{1}{q}\right)}$ via comparison to $\left(\mathbb{E}W_2^2(\mu, \mu_n)\right)^{1/2}$ using Remark \ref{SRWAlternateForms} and established bounds for the Wasserstein distance such as from \cite[Theorem 3.1]{Lei}, which is stronger in low dimensions.  
\end{remark}

\subsection{Some definitions and notations}\label{coupling}
Throughout this paper, the following will be used:
\begin{itemize}
    \item For any separable Hilbert space $\mathcal{K}$: 
    \begin{itemize} \item For $x, y \in \mathcal{K}$, $(x \otimes y)$ denotes the operator on $\mathcal{K}$ defined by $v \mapsto \langle y, v \rangle x.$ \smallskip
    \item If $A: \mathcal{K} \rightarrow \mathcal{K}$ is a bounded linear operator, then $\left\lVert A \right\rVert_{\mathrm{op}} \coloneq  \sup_{\|x\| \leq 1} \|Ax\|$ and $A^*$ denotes its adjoint, i.e., \[\langle Ax, y\rangle = \langle x, A^*y\rangle, \quad  \, x,y \in \mathcal{K}.\] We say $A$ is trace class (or 1-Schatten operator) if $\mathrm{Tr}\left((A^*A)^{1/2}\right) < \infty.$\smallskip
    \item $\mathcal{G}_k(\mathcal{K})$ denotes the collection of $k$-dimensional subspaces of $\mathcal{K}$.\smallskip
    
    \item $\mathcal{P}(\mathcal{K})$ denotes the collection of all probability measures on $\mathcal{K}$ and \[\mathcal{P}_2(\mathcal{K}) \coloneqq \left\{ \mu \in \mathcal{P}(\mathcal{K}) : \int_{\mathcal{K}}\|x\|^2 d\mu(x) < \infty \right\}.\]
    
    \item If $X$ is a random vector in $\mathcal{K}$ with $\mathbb{E}\|X\| < \infty$, then $\mathbb{E}X$ is the vector in $\mathcal{K}$ such that \[\langle \mathbb{E}X, y \rangle = \mathbb{E}\langle X, y \rangle, \quad y \in \mathcal{K};\] and if for some $\mu \in \mathcal{P}_2(\mathcal{K})$, $X$ is distributed according to $\mu$, then $\int_{\mathcal{K}} (x \otimes x)\, d\mu(x)$ is the positive semi-definite 1-Schatten operator on $\mathcal{K}$ defined by 
\begin{equation*}
v \mapsto \mathbb{E}(\langle X , v \rangle X), 
\end{equation*} with trace $\int_{\mathcal{K}} \|x\|^2 \, d\mu(x) < \infty.$ 
    \end{itemize}
    \item $\mathcal{H}$ denotes a \textbf{separable infinite dimensional Hilbert space}. 
    \item For $q > 2$, \[\mathcal{P}_q :=\left\{ \mu\in\mathcal{P}(\mathcal{H}) : \int_{\mathcal{H}}\|x\|^q\,d\mu(x)\leq 1 \right\}.\] Note that by Jensen's inequality, $\mathcal{P}_q \subset \mathcal{P}_2(\mathcal{H})$.
    \item For $d \in \mathbb{N}$, $T$ denotes a \textbf{linear isometric embedding from $\mathbb{R}^d$ into $\mathcal{H}$}.
    \item $B_2^d(x_0,\epsilon) \coloneq \{x \in \mathbb{R}^d: \|x-x_0\|_2\leq \epsilon\}$ is the \textbf{closed ball of radius $\epsilon > 0$ around $x_0$ in $\mathbb{R}^d$}, where $\|\cdot\|_2$ is the Euclidean norm on $\mathbb{R}^d$. 
    \item If $E, F$ are measurable sets, $\mu$ is a probability measure on $E$, and $Q:E \rightarrow F$, then $Q_{\#}\mu$ denotes the \textbf{pushforward probability measure} on $F$, i.e., \[Q_{\#}\mu(A) = \mu(Q^{-1}(A)),\] for all measurable $A \subset F$.
    \item If $\mu, \nu$ are two probability measures on a set $E$, then $\Pi(\mu, \nu)$ denotes \textbf{the coupling of $\mu$ and $\nu$}, i.e., the collection of all probability measures on $E \times E$ with marginal distributions $\mu$ and $\nu$ with respect to the first and second components respectively.
    \item For any probability measure $\mu$ on a set $E$, we will define $\mu_n = \frac{1}{n}\summ{1}{n}\delta_{X_i}$ to be the \textbf{empirical measure} of $\mu$, where $X_i$'s are independent samples of $\mu$ in $E$. 
\end{itemize}

\subsection{Organization of this paper}
In the rest of this paper, we prove Theorem \ref{Mainthm} after stating some results required for the proof.
\begin{itemize}
    \item In Section 2, we state some results that are useful for the rest of this paper.
    \item In Section 3, we prove the upper bound in Theorem \ref{Mainthm}. 
    \item In Section 4, we prove the lower bound in Theorem \ref{Mainthm}.
\end{itemize}
\subsection{Idea of Proof of Theorem \ref{Mainthm}}
For quick reference, we outline the proof of Theorem \ref{Mainthm} here. Both bounds are achieved by analyzing embeddings into logarithmic-dimensional subspaces. Fix $q > 2$, $k \in \mathbb{N}$, and $n \geq 3$.
\begin{itemize}
    \item \textbf{Upper bound.} Fix $\mu \in \mathcal{P}_q$, and let $\lambda_1 \geq \lambda_2 \geq \cdots \geq 0$ be the eigenvalues of $\int_{\mathcal{H}} x\otimes x\, d\mu(x)$, so that $\sum_{i=1}^{\infty} \lambda_i \leq 1$. If $k \geq \log n$, the trivial bound $\mathbb{E}S_k^2(\mu, \mu_n) \leq 2$, obtained by considering the independent coupling of $\mu$ and $\mu_n$, proves the upper bound, since then $\sqrt{\min\{k, \log n\}\log \log n/\log n}$ is bounded below by $\sqrt{\log \log 3}$. 
    
    If $k < \log n$, since $S_k \leq \sqrt{k}\,S_1$ (Lemma \ref{rootk}), it suffices to bound $S_1$. Let $P$ be the orthogonal projection onto the span of the eigenvectors corresponding to $\lambda_1, \ldots, \lambda_t$ for some $t$. By the triangle inequality,
\[
\left(\mathbb{E}S_1^2(\mu, \mu_n)\right)^{1/2} \leq S_1(\mu, P_{\#}\mu) + \left(\mathbb{E}S_1^2(P_{\#}\mu, P_{\#}\mu_n)\right)^{1/2} + \left(\mathbb{E}S_1^2(P_{\#}\mu_n, \mu_n)\right)^{1/2}.
\]
The middle term is bounded by $C_q\left(n^{-1/t} + n^{-(q-2)/2q}\right)\log n$ via Remark \ref{SRWAlternateForms} and an adaptation of a result of Lei \cite[Theorem 3.1]{Lei} on finite-dimensional empirical convergence in $W_2$ under finite $q$-th moment. The remaining projection error terms are controlled using the natural coupling between $X \sim \mu$ and $PX$ and the fact that $\lambda_{t+1} \leq \lambda_t \leq 1/t$, to yield a bound of the form $1/\sqrt{t}$, plus an additional sampling error for the third term which we bound by $Cn^{-(q-2)/2q}\sqrt{\log n}$ using an extension of \cite[Theorem 5.48]{Vershynin2} to infinite dimensions. Finally, a suitable choice of $t \asymp \log n/\log\log n$ balances $1/\sqrt{t}$ against $n^{-1/t}\log n$, and the resulting rate $\sqrt{\log\log n/\log n}$ dominates the terms decaying polynomially in $n$. This provides the upper bound for $S_1$, and hence $\sqrt{k\log \log n/\log n}$ for $S_k$ up to a constant depending only on $q$.

\item \textbf{Lower bound.} Let $\nu$ be the uniform distribution on the Euclidean unit ball of $\mathbb{R}^d$ with $d = \lceil \log n \rceil$, and let $\widetilde{k}= \min\{k, d\}$. Keeping track of the constant in existing results (such as \cite{Dudley}) by using the fact that $n$ balls of radius $\epsilon = (2n)^{-1/d}$ around the sample points carry $\nu$-mass at most $n\epsilon^d = 1/2$, we obtain $W_2(\nu, \nu_n) \gtrsim n^{-1/d}$. Then, for the pushforward $T_{\#}\nu \in \mathcal{P}_q$ of $\nu$ under the isometric embedding $T$ into $\mathcal{H}$, we obtain
\[\left(\mathbb{E}S_k^2(T_{\#}\nu, T_{\#}\nu_n)\right)^{1/2} \geq \left(\mathbb{E}S_{\widetilde{k}}^2(T_{\#}\nu, T_{\#}\nu_n)\right)^{1/2} \geq \left(\mathbb{E}S_{\widetilde{k}}^2(\nu, \nu_n)\right)^{1/2} \gtrsim \sqrt{\frac{{\widetilde{k}}}{{d}}}n^{-1/d} \gtrsim \sqrt{\frac{\min\{k, \log n\}}{\log n}}, \]  where we use the comparisons $S_k \geq S_{\widetilde{k}}$ (Lemma \ref{rootk}), and $S_{\widetilde{k}} \geq \sqrt{\widetilde{k}/d}\cdot W_2$ in $\mathbb{R}^d$ (Remark \ref{SRWAlternateForms}). 

\end{itemize}

\subsection{Organization of this paper}
The rest of this paper is organized as follows.
\begin{itemize}
    \item In Section \ref{prelim}, we state some results that are useful for the rest of this paper. 
    \item In Section \ref{upperboundsec}, we prove the upper bound in Theorem \ref{Mainthm}. 
    \item In Section \ref{lowerboundsec}, we prove the lower bound in Theorem \ref{Mainthm}.
    \item In Section \ref{furtherresults}, we show that Theorem \ref{Mainthm} may be improved under stronger assumptions. 
\end{itemize}

\section{Some useful preliminary results}\label{prelim}

\subsection{Some results on subspace robust Wasserstein distance}
The following result follows from \cite[Proposition 3]{Paty} for $\mathbb{R}^d$. We extend it to any separable Hilbert space by adapting arguments used to prove \cite[Proposition 2]{Paty}. 
\begin{lemma}\label{rootk}
    Let $\mathcal{K}$ be a separable Hilbert space and let $\mu, \nu \in \mathcal{P}(\mathcal{K})$. Then, we have
    \[S_{\widetilde{k}}(\mu, \nu) \leq S_k(\mu, \nu) \leq \sqrt{k}\cdot S_1(\mu, \nu), \quad 1 \leq \widetilde{k} \leq k \leq \mathrm{dim}(\mathcal{K}).\] 
\end{lemma}
\begin{proof}
     
    The first inequality holds because every $\widetilde{k}$-dimensional subspace of $\mathcal{K}$ is contained in a $k$-dimensional subspace of $\mathcal{K}$ since $\widetilde{k}\leq k$. 
    
    To prove the second inequality $S_k(\mu, \nu) \leq \sqrt{k}\cdot S_1(\mu, \nu)$, fix $\pi \in \Pi(\mu, \nu)$ and $E \in \mathcal{G}_k(\mathcal{K})$. Let $\{w_1, \ldots, w_k\}$ be an orthonormal basis of $E$. Observe that 
    \begin{eqnarray}\label{Projection}
        \int \|P_E(x-y)\|^2\text{ }d\pi(x,y) &=& \int \summ{1}{k} \langle w_i, (x-y)\rangle^2 \text{ }d\pi(x,y)\nonumber \\ &\leq& k \cdot \sup_{\|w\|\leq 1} \int \langle w, (x-y)\rangle^2 \text{ }d\pi(x,y). 
    \end{eqnarray} 
    Then, using (\ref{Projection}) and Remark \ref{S_1}, we have \begin{align*} S_k^2(\mu, \nu) &= \inf_{\pi \in \Pi(\mu, \nu)}\sup_{E \in \mathcal{G}_k(\mathcal{K})} \int \|P_E(x-y)\|^2\text{ }d\pi(x,y) \\ &\leq k \cdot \inf_{\pi \in \Pi(\mu, \nu)}\sup_{\|w\|\leq 1}\int \langle w, (x-y)\rangle^2 \text{ }d\pi(x,y) = k \cdot S_1^2(\mu, \nu).\qedhere
    \end{align*}
\end{proof}

The following result allows us to compare the SRW distance between measures on $\mathbb{R}^d$ with that between its pushforwards under an isometric embedding into $\mathcal{H}$.
\begin{proposition}\label{prop}
    Let $T$ be as defined in Section \ref{coupling}. Let $\mu, \nu \in \mathcal{P}_2(\mathbb{R}^d)$. 
    Then, for any $1 \leq k \leq d,$ \[S_k(\mu, \nu) \leq S_k(T_{\#}\mu, T_{\#}\nu).\] 
\end{proposition}
\begin{proof}
Define $\widetilde{T}: \mathbb{R}^d \times \mathbb{R}^d \rightarrow \mathcal{H}\times \mathcal{H}$ as $(x,y) \mapsto (T(x), T(y))$. Note that since $T$ is a linear isometry, it has a continuous inverse on its range $T(\mathbb{R}^d)$, and $T_{\#}\mu, T_{\#}\nu$ are supported on $T(\mathbb{R}^d)$. It follows that
\begin{equation}\label{claim}
    \Pi(T_{\#}\mu, T_{\#}\nu) = \{\widetilde{T}_{\#}\pi: \pi \in \Pi(\mu, \nu) \}.
\end{equation}

Let $P_E$ denote the orthogonal projection onto $E$ in the ambient space ($\mathcal{H}$ or $\mathbb{R}^d$). We have
    \begin{eqnarray*}
        S_k(T_{\#}\mu, T_{\#}\nu) &=&\inf_{\gamma \in \Pi(T_{\#}\mu, T_{\#}\nu)} \sup_{E\in\mathcal{G}_k(\mathcal{H})} \left(\int_{\mathcal{H} \times \mathcal{H}} \| P_E(x-y) \|^2 \text{ }d\gamma(x,y) \right)^{1/2} \\ &=&
        \inf_{\gamma \in \Pi(T_{\#}\mu, T_{\#}\nu)} \sup_{E\in\mathcal{G}_k(\mathcal{H})} \left(\int_{T(\mathbb{R}^d) \times T(\mathbb{R}^d)} \| P_E(x-y) \|^2 \text{ }d\gamma(x,y) \right)^{1/2} 
        \\&\geq&
       \inf_{\gamma \in \Pi(T_{\#}\mu, T_{\#}\nu)}  \sup_{E\in\mathcal{G}_k(T(\mathbb{R}^d))} \left(\int_{T(\mathbb{R}^d) \times T(\mathbb{R}^d)} \| P_E(x-y) \|^2 \text{ }d\gamma(x,y) \right)^{1/2} \\ &=& \inf_{\pi \in \Pi(\mu, \nu)} \sup_{E\in\mathcal{G}_k({T(\mathbb{R}^d)})} \left(\int_{T(\mathbb{R}^d) \times T(\mathbb{R}^d)} \| P_E(x-y) \|^2\text{ } d(\widetilde{T}_{\#}\pi)(x,y) \right)^{1/2} \\ &=&
        \inf_{\pi \in \Pi(\mu, \nu)} \sup_{E \in \mathcal{G}_k(\mathbb{R}^d)} \left(\int_{T(\mathbb{R}^d) \times T(\mathbb{R}^d)} \| P_{T(E)}(x-y) \|^2\text{ } d(\widetilde{T}_{\#}\pi)(x,y) \right)^{1/2} \\ &=&
        \inf_{\pi \in \Pi(\mu, \nu)} \sup_{E \in \mathcal{G}_k(\mathbb{R}^d)} \left(\int_{\mathbb{R}^d \times \mathbb{R}^d} \| P_{T(E)}T(u-v) \|^2\text{ } d\pi(u,v) \right)^{1/2} \\ &=&\inf_{\pi \in \Pi(\mu, \nu)} \sup_{E \in \mathcal{G}_k(\mathbb{R}^d)} \left(\int_{\mathbb{R}^d \times \mathbb{R}^d} \| T P_{E}(u-v) \|^2\text{ } d\pi(u,v) \right)^{1/2}   \\ &=&\inf_{\pi \in \Pi(\mu, \nu)} \sup_{E \in \mathcal{G}_k(\mathbb{R}^d)} \left(\int_{\mathbb{R}^d \times \mathbb{R}^d} \|P_{E}(u-v) \|^2\text{ } d\pi(u,v) \right)^{1/2} \\&=&
        S_k(\mu, \nu), 
    \end{eqnarray*}
        where the first and last steps follow from (\ref{SRW}), the second step follows from the fact that every $\gamma \in \Pi(T_{\#}\mu, T_{\#}\nu)$ is supported on $T(\mathbb{R}^d) \times T(\mathbb{R}^d)$, the third step follows since $\mathcal{G}_k(T(\mathbb{R}^d)) \subset \mathcal{G}_k(\mathcal{H})$, the fourth step follows from (\ref{claim}), the fifth step follows since $T$ induces a bijection between $\mathcal{G}_k(\mathbb{R}^d)$ and $\mathcal{G}_k(T(\mathbb{R}^d))$, the sixth step follows from a change of variables, and both the seventh and eighth steps follow because $T$ is an isometry, thereby preserving inner products and norms. This completes the proof. 
\end{proof}

The next result follows from \cite[Lemma 2]{Paty} with $k=1$ in $\mathbb{R}^d$. We include a short direct proof valid in any separable Hilbert space.
\begin{lemma}\label{opsrw}
    Suppose that $\mathcal{K}$ is a separable Hilbert space and $\mu \in \mathcal{P}_2(\mathcal{K})$. Let $P$ be the orthogonal projection onto a subspace of $\mathcal{K}$. Then, 
    \begin{equation*}
        S_1(\mu, P_\#\mu) \leq \left\|(I-P)\left(\int_{\mathcal{K}}(x \otimes x)\text{ }d\mu(x)\right)(I-P)\right\|^{1/2}_{\mathrm{op}}.
    \end{equation*}
\end{lemma}
\begin{proof}
    Let $A: \mathcal{K} \longrightarrow \mathcal{K} \times \mathcal{K}$ be defined by $x \mapsto (x, Px)$. Let $\pi = A_{\#}\mu$. Then, $\pi \in \Pi(\mu, P_\#\mu)$ and using Remark \ref{S_1}, we have,
    \begin{align*}
        S_1(\mu, P_\#\mu) &\leq \sup_{\left\lVert w \right\rVert \leq 1} \left( \int_{\mathcal{K}\times \mathcal{K}} \left|\langle w,x-y\rangle \right|^2 d\pi(x,y) \right)^{1/2}
        = \sup_{\|w\|\leq 1} \left( \int_{\mathcal{K}}|\langle w, x-Px\rangle|^2\text{ }d\mu(x) \right) ^{1/2}\\
    &= \sup_{\left\lVert w \right\rVert \leq 1} \left\langle \left[(I-P)\left(\int_{\mathcal{K}}(x \otimes x)\text{ }d\mu(x)\right) (I-P)\right] w, w \right\rangle^{1/2}
    \\&= \left\|(I-P)\left(\int_{\mathcal{K}}(x \otimes x)\text{ }d\mu(x)\right)(I-P)\right\|_{\mathrm{op}}^{1/2}.\qedhere
    \end{align*}
\end{proof}
\subsection{Bound on Expected Sample Covariance}
The following result extends \cite[Theorem 5.48]{Vershynin2} to the infinite dimensional case. 
\begin{lemma}\label{marchresult}
Let $q > 2$. Suppose that $\mu \in \mathcal{P}_q$. Let $X_1, X_2, \ldots, X_n$ be i.i.d. random vectors in $\mathcal{H}$ sampled according to $\mu$. Then, for $n \geq 3$,
    \begin{equation*}
        \mathbb{E}\left\|\summ{1}{n} X_i \otimes X_i \right\|_{\mathrm{op}} \leq 2n \left\| \mathbb{E} (X_1 \otimes X_1) \right\|_{\mathrm{op}}+ 2C^2n^{2/q} \log n,
    \end{equation*}
    where $C> 0$ is a universal constant. 
\end{lemma}

\begin{proof}
Let $\Sigma \coloneq \mathbb{E}(X_1 \otimes X_1)$, $S\coloneq \summ{1}{n}\, X_i \otimes X_i$, and let $P_m$ be the orthogonal projection onto $\mathcal{H}_m \coloneq \mathrm{Span}(\{e_1, \ldots, e_m\})$ which is isometric to $\mathbb{R}^m$, for an orthonormal basis $\{e_1, e_2, \ldots \}$ of $\mathcal{H}$. Let 
\[A_m = \left\|S-P_mSP_m \right\|_{\mathrm{op}},\quad B_m = \left\|P_mSP_m - nP_m\Sigma P_m \right\|_{\mathrm{op}}.\] 
For any $m \geq n$, applying \cite[Theorem 5.48]{Vershynin2} to the i.i.d. random vectors $P_mX_1, \cdots, P_mX_n$ in $\mathcal{H}_m$, distributed according to $(P_m)_{\#}\mu$, we get
\begin{equation}\label{versh}
    \mathbb{E}B_m \leq n\cdot \max \left\{\|P_m\Sigma P_m\|_{\mathrm{op}}^{1/2}\eta, \eta^2\right\};\quad \eta \coloneq C\sqrt{\frac{\mathbb{E}\max_{i\leq n}\|P_m X_i\|^2 \cdot \log \min\{n, m\}}{n}}, 
\end{equation} where $C>0$ is a universal constant. Note that $\eta \leq C\sqrt{n^{(2/q)-1}\log n}$ since $m \geq n$, and $\mu\in \mathcal{P}_q$ implies $\mathbb{E}\underset{i\leq n}{\max}\|P_m X_i\|^2 \leq \mathbb{E}\underset{i\leq n}{\max}\|X_i\|^{2} \leq \left(\mathbb{E}\summ{1}{n}\,\|X_i\|^q\right)^{2/q}\leq n^{2/q}$ by Jensen's inequality. Then, from (\ref{versh}), using $\max\{a,b\} \leq a+b$, and $\sqrt{ab} \leq \left(\frac{a}{2}+\frac{b}{2}\right)$, we have for any $m \geq n$,
\begin{eqnarray}\label{Bm}
    \mathbb{E}B_m &\leq& \sqrt{n\|P_m\Sigma P_m\|_{\mathrm{op}}}\cdot C\sqrt{n^{(2/q)}\log n}+ C^2n^{2/q}\log n \nonumber \\ 
    &\leq& \frac{1}{2}n\|P_m \Sigma P_m\|_{\mathrm{op}} + \frac{3C^2}{2}n^{2/q} \log n \leq \frac{1}{2}n\|\Sigma \|_{\mathrm{op}} + \frac{3C^2}{2}n^{2/q} \log n. 
\end{eqnarray}
Now, for each $m$, using the bound $\|x\otimes x - y \otimes y\|_{\mathrm{op}}\leq (\|x\|+\|y\|)\|x-y\|$, we get \[A_m \leq\summ{1}{n}\|X_i\otimes X_i - (P_mX_i)\otimes(P_mX_i)\|_{\mathrm{op}}\leq \summ{1}{n}2\|X_i\|\|(I-P_m)X_i\| \leq \summ{1}{n}2\|X_i\|^2,\]
where the right-hand side has expectation $\mathbb{E}\summ{1}{n}2\|X_i\|^2 = 2n\mathbb{E}\|X_1\|^2 < \infty$. Also note that $\summ{1}{n}2\|X_i\|\|(I-P_m)X_i\| \rightarrow 0$ as $m \rightarrow \infty$ since $(I-P_m)x \rightarrow 0$ for every $x \in \mathcal{H}$ as $m \rightarrow \infty$. Then, by the dominated convergence theorem, we have $\lim_{m \rightarrow \infty}\mathbb{E}A_m = 0$. By the triangle inequality, we have for any $m \geq n$, 
\[\mathbb{E}\left\|\summ{1}{n}\, X_i \otimes X_i\right\|_{\mathrm{op}} \leq \mathbb{E}A_m + \mathbb{E}B_m + n\|P_m\Sigma P_m\|_{\mathrm{op}} \leq \mathbb{E}A_m + \mathbb{E}B_m + n\|\Sigma\|_{\mathrm{op}}.\] Finally, since $\lim_{m \rightarrow \infty}\mathbb{E}A_m = 0$, using (\ref{Bm}) and letting $m \rightarrow \infty$ gives
\[\mathbb{E}\left\|\summ{1}{n}\, X_i \otimes X_i\right\|_{\mathrm{op}} \leq \frac{3}{2}n\|\Sigma\|_{\mathrm{op}} + \frac{3C^2}{2}n^{2/q} \log n \leq 2n\|\mathbb{E}(X_1 \otimes X_1)\|_{\mathrm{op}}+ 2C^2n^{2/q}\log n.\qedhere\]

\end{proof}

\subsection{Some results on Wasserstein distance}
The following result is proved within the proof of \cite[Theorem 3.1]{Lei}. 
\begin{lemma}(\cite{Lei}, proof of Theorem 3.1)\label{Cn^-1/t}
Let $q > 2$, $d \geq 5$ and $n \geq 3$. Let $\mu \in \mathcal{P}(\mathbb{R}^d)$ such that $\int_{\mathbb{R}^d}\lVert x \rVert^q \text{ }d\mu(x) \leq 1$ and let $\mu_n$ be the empirical distribution of $\mu$. Then, we have: \[
\left(\mathbb{E}W_2^2(\mu, \mu_n)\right)^{1/2} \leq C_{q}n^{-\min{\left\{1/d, (q-2)/2q \right\}}}\cdot \log n,
\] where $C_q> 0$ depends only on $q$. 
\end{lemma}
\begin{proof}
    Follow the proof of \cite[Theorem 3.1]{Lei} (see \cite[Appendix C]{Lei}) without the final application of Jensen's inequality and bound the logarithm factor as remarked in \cite[Remark 2]{Lei} to obtain: \[\mathbb{E}W_p^p(\mu, \mu_n) \leq c_{p,q} M_q^p \, n^{ - \min \left\{ \frac{1}{2}, \frac{q-p}{q}, \frac{p}{d} \right\}} (\log n)^{2},\] where $c_{p,q} > 0$ depends only on $p, q$ and $M_q^q = \int_{\mathbb{R}^d}\|x\|^q \, d\mu(x) \leq 1.$
    Take $p = 2.$ Since $d \geq 5$, we have $\min \left\{ \frac{1}{2}, \frac{q-2}{q}, \frac{2}{d} \right\} = \min{\left\{ \frac{2}{d}, \frac{q-2}{q} \right\}}$. Taking square roots then proves the result with $C_q = \sqrt{c_{2,q}}.$
\end{proof}

Lower bounds of order $n^{-1/d}$ for the Wasserstein distance between an absolutely continuous distribution on $\mathbb{R}^d$ and any distribution supported on $n$ points are well-known (see, e.g., \cite[Proposition 2.1]{Dudley}, one of the earliest mentions, or \cite[Proposition 4.7.7]{Panaretos} for a more explicit proof). Here, we consider the uniform distribution on the unit ball and keep track of the constant involved to ensure it is dimension-independent. 

\begin{lemma}\label{unif}
    Let $\nu$ be the uniform probability measure on $B^d_2(0,1) \subset \mathbb{R}^d$. Then, 
    \[\left(\mathbb{E}W_2^2(\nu,\nu_n)\right)^{1/2} \geq \frac{1}{2\sqrt{2}}n^{-1/d}. \]  
\end{lemma}
 
\begin{proof}
    Let $M = \{x_1, \ldots, x_n\} \subset \mathbb{R}^d.$ Fix $\epsilon = (2n)^{-1/d}$ and define $A = \overset{n}{\underset{i =1}{\bigcup}}B_2^d(x_i, \epsilon)$. For any $\pi \in \Pi(\nu,\frac{1}{n}\summ{1}{n}\delta_{x_i})$, we have
   \begin{eqnarray}\label{unifbdproof}
        \int_{\mathbb{R}^d\times M}\|x-y\|^2 d\pi(x,y) &\geq& \int_{A^c \times M}\|x-y\|^2 d\pi(x,y) \nonumber \\&\geq& \epsilon^2 \pi(A^c \times M) = \epsilon^2 \nu(A^c)\nonumber \\ &\geq&\epsilon^2\frac{1}{2} = \frac{1}{2}(2n)^{-2/d} \geq \frac{1}{8}n^{-2/d},
    \end{eqnarray} where the third inequality follows because $\nu( A^c ) \geq 1 - \summ{1}{n}\, \nu \left( B_2^d(x_i, \epsilon) \right) \geq 1 - n\epsilon^d = \frac{1}{2}.$ Since (\ref{unifbdproof}) holds for every coupling, taking the infimum over $\Pi\big(\nu, \frac1n\sum_{i=1}^n\delta_{x_i}\big)$ gives
\[W_2^2\left(\nu, \frac{1}{n}\sum_{i=1}^n\delta_{x_i}\right) \geq \frac{1}{8}\,n^{-2/d}\]
for any $x_1,\dots,x_n \in \mathbb{R}^d$. By randomizing $x_1, \ldots, x_n$ according to $\nu$ so that $\frac{1}{n}\summ{1}{n}\delta_{X_i} = \nu_n$, we obtain $\mathbb{E}\,W_2^2(\nu,\nu_n) \ge \frac18 n^{-2/d}$, where the expectation is over $X_1, \ldots, X_n$.
\end{proof}

\section{Proof of upper bound in Theorem \ref{Mainthm}}\label{upperboundsec}
The goal of this section is to prove the upper bound in Theorem \ref{Mainthm}, i.e., for given $q > 2, k \in \mathbb{N},$ and $n \geq 3,$
\begin{equation}\label{upperbd}
  \sup_{\mu\in\mathcal{P}_q}\left( \mathbb{E}S_k^2(\mu, \mu_n) \right)^{1/2} \leq C_q\sqrt{\frac{{\min\{k, \log n\}\log \log n}}{{\log n}}},  
\end{equation}
where $C_q>0$ depends only on $q$.
Observe that the following hold, for any $\mu \in \mathcal{P}_q$.
\begin{itemize}
\item $\int_{\mathcal{H}}(x \otimes x)\text{ }d\mu(x)$  is a 1-Schatten operator with trace \begin{equation}\label{tracebound}\mathrm{Tr}\left(\int_{\mathcal{H}}(x \otimes x)\text{ } d\mu(x)\right) = \int_{\mathcal{H}}\|x\|^2\text{ }d\mu(x) \leq \left(\int_{\mathcal{H}}\|x\|^q\text{ }d\mu(x)\right)^{2/q} \leq 1. \end{equation}
\item $\mathbb{E}S_k^2(\mu, \mu_n) \leq 2$. Indeed, let $X_1, \ldots, X_n$ be i.i.d. samples of $\mu$ so that $\mu_n = \frac{1}{n}\summ{1}{n} \delta_{X_i}.$ For any $E \in \mathcal{G}_k(\mathcal{H})$, let $P_E$ be the orthogonal projection onto $E$. Taking $\pi= \mu \otimes \mu_n$ to be the independent coupling in $\Pi(\mu, \mu_n),$ we thus obtain:
\begin{eqnarray}\label{Sk<Wk}
   \mathbb{E}S_k^2(\mu, \mu_n) &\leq& \mathbb{E}\left\{\sup_{E \in \mathcal{G}_k(\mathcal{H})} \int_{\mathcal{H}\times \mathcal{H}}\|P_E (x-y)\|^2\text{ }d\pi(x,y)\right\} \nonumber \\ &\leq& \mathbb{E}\left\{\int_{\mathcal{H}\times \mathcal{H}}\|x-y\|^2\text{ }d\pi(x,y)\right\}\nonumber\\ &=& \frac{1}{n}\summ{1}{n}\left\{ \mathbb{E}\left[\int_{\mathcal{H}} \|x-X_i\|^2\text{ }d\mu(x)\right]\right\}\nonumber \\ &=&  \frac{1}{n} \summ{1}{n}\left\{\displaystyle\int_{\mathcal{H}}\left(\int_{\mathcal{H}}\|x-y\|^2 d\mu(x)\right) d\mu(y)\right\}\nonumber \\ &=& 2\int_{\mathcal{H}}\|x\|^2 d\mu(x) - 2\int_{\mathcal{H}}\int_{\mathcal{H}}\langle x, y \rangle\text{ }d\mu(x)d\mu(y) \nonumber \\
    &\leq& 2\left(\int_{\mathcal{H}}\|x\|^q d\mu(x)\right)^{2/q} - 2\left\langle \int_{\mathcal{H}} x \text{ }d\mu(x),\int_{\mathcal{H}} y\text{ }d\mu(y) \right\rangle\text{ } \nonumber \\ &\leq&  2-2\left\|\int_{\mathcal{H}}x\text{ }d\mu(x) \right\|^2 \leq 2,
\end{eqnarray}
where the first and third steps follow from the definitions of $S_k^2(\mu, \mu_n)$ and $\pi$, the second step holds because for any $E \in \mathcal{G}_k(\mathcal{H}),$ $\|P_Ex\| \leq \|x\|$ for all $x \in \mathcal{H}$, the fourth step follows because $X_i$'s are identically distributed according to $\mu$, the fifth step follows by expansion of $\|x-y\|^2$, the sixth step follows from Jensen's inequality and \cite[Theorem 17.14]{Hilbert}, and the seventh step follows because $\mu \in \mathcal{P}_q$. 
\end{itemize}
We now state the following Proposition \ref{prop2} which we prove in Part I and use in Part II to complete the proof of the upper bound.  
\begin{proposition}\label{prop2}
Let $q >2$ and $k \in \mathbb{N}$. Let $\mu \in \mathcal{P}_q$. If $\lambda_1 \geq \lambda_2 \geq \ldots \geq 0$ denote the eigenvalues of $\int_{\mathcal{H}}(x \otimes x)\text{ } d\mu(x)$, then given $n \geq 3$, we have
\begin{equation*}
  \left(\mathbb{E}S_k^2(\mu, \mu_n)\right)^{1/2}\leq \inf_{t \geq 5}\left\{  \sqrt{k}\left(C_1\sqrt{\lambda_t}+\frac{C_2(q)\log n}{n^{(q-2)/2q}}+C_3(q) \frac{\log n}{n^{1/t}} \right)\right\} ,
\end{equation*}
where $C_1> 0$ is a universal constant and $C_2(q), C_3(q)>0$ depend only on $q$.  
\end{proposition}

\subsection{Part I: Proof of Proposition \ref{prop2}}

\begin{proof}
Note that by Lemma \ref{rootk}, it is sufficient to prove that \begin{equation}\label{prop2statementforS1}\left(\mathbb{E}S_1^2(\mu, \mu_n)\right)^{1/2}\leq \inf_{t \geq 5}\left\{C_1\sqrt{\lambda_t}+\frac{C_2(q)\log n}{n^{(q-2)/2q}}+C_3(q) \frac{\log n}{n^{1/t}} \right\}.\end{equation}

Fix $t \geq 5$. Let $P$ be the orthogonal projection onto the subspace spanned by $t$ orthonormal eigenvectors that correspond to the $t$ largest eigenvalues $\lambda_1, \ldots, \lambda_t$. Note that $\mu_n \in \mathcal{P}_2(\mathcal{H})$ since $\mu_n$ is finitely supported. Since $S_1(\cdot, \cdot)$ is a metric, 
\[
S_1(\mu, \mu_n) \leq S_1(\mu, P_{\#}\mu) + S_1(P_{\#}\mu, P_{\#}\mu_n) + S_1(P_{\#}\mu_n,\mu_n).
\]
Also, since $(\mathbb{E}(\cdot)^2)^{1/2}$ is a norm, by the triangle inequality, we get: 
\begin{equation}\label{triangleineqsplit}
        \left(\mathbb{E}S_1^2(\mu,\mu_n)\right)^{1/2}
         \leq  S_1(\mu, P_{\#}\mu) +\left(\mathbb{E}S_1^2(P_{\#}\mu, P_{\#}\mu_n) \right)^{1/2} + \left(\mathbb{E}S_1^2(P_{\#}\mu_n,\mu_n)\right)^{1/2}.
\end{equation}
We now bound each of the three terms on the right hand side of (\ref{triangleineqsplit}). To bound the first term, we apply Lemma \ref{opsrw} and obtain
\begin{equation}\label{firstterm}
    S_1(\mu, P_{\#}\mu) \leq \left\|(I-P)\left(\int_{\mathcal{H}}(x \otimes x)\text{ }d\mu(x)\right)(I-P)\right\|_{\mathrm{op}}^{1/2} =  \sqrt{\lambda_{t+1}} \leq \sqrt{\lambda_t}.
\end{equation}

To bound the second term on the right hand side of (\ref{triangleineqsplit}), observe that since for any $x \in \mathcal{H},$ $\|Px\| \leq \|x\|$, we have $\int_{\mathcal{H}} \|Px\|^q \text{ } d\mu(x) \leq \int_{\mathcal{H}} \|x\|^q\text{ }d\mu(x) \leq 1$. Additionally, observe that $P_{\#}\mu_n$ is an empirical measure of $P_{\#}\mu$, and $P_{\#}\mu$ is a measure on a $t$-dimensional Hilbert space which is isometric to $(\mathbb{R}^t, \|\cdot \|_2)$. So, by Remark \ref{SRWAlternateForms} and Lemma \ref{Cn^-1/t} which we can apply since $t \geq 5$, we have: 
\begin{eqnarray}\label{secondterm}
    \left( \mathbb{E}S_1^2(P_{\#}\mu, P_{\#}\mu_n)\right)^{1/2} &\leq& (\mathbb{E}W_2^2(P_{\#}\mu, P_{\#}\mu_n))^{1/2}\\\nonumber &\leq& C_3(q) n^{-\min{\left\{1/t, (q-2)/2q \right\}}}\cdot \log n\\\nonumber &\leq& C_3(q)(n^{-1/t}\cdot\log n+n^{-(q-2)/2q}\cdot\log n),
\end{eqnarray}
where $C_3(q) > 0$ depends only on $q$. 

We now bound the third term on the right hand side of (\ref{triangleineqsplit}). Let $X_1, \ldots, X_n$ be independent samples of $\mu$ so that $\mu_n = \frac{1}{n}\summ{1}{n}\delta_{X_i}$. Since $\|(I-P)x\| \leq \|x\|$ for any $x \in \mathcal{H}$, \[\displaystyle\int_{\mathcal{H}}\left\lVert(I-P)x\right\rVert^q \, d\mu(x) \leq \int_{\mathcal{H}} \|x\|^q \, d\mu(x) \leq 1.\] Then, applying Lemma \ref{opsrw} and Lemma \ref{marchresult} to $(I-P)X_1, \ldots, (I-P)X_n$, we obtain the following:
\begin{align}\label{thirdterm}
    &\left( \mathbb{E}S_1^2(P_{\#}\mu_n, \mu_n) \right)^{1/2} \nonumber\\ \leq& \left(\mathbb{E}\left\|(I-P) \left( \frac{1}{n}\summ{1}{n} X_i \otimes X_i \right) (I-P)\right\|_{\mathrm{op}} \right)^{\frac{1}{2}} \nonumber \\
    \leq& \text{ } n^{-1/2} \cdot \left(2n \left\|\mathbb{E}\left[ (I-P)(X_1 \otimes X_1)(I-P) \right]\right\|_{\mathrm{op}} + 2C^2n^{2/q}\log n\right)^{\frac{1}{2}} \nonumber\\
    \leq& \sqrt{2}\left\|(I-P)\left(\int_{\mathcal{H}}(x \otimes x)\text{ }d\mu(x)\right)(I-P)\right\|_{\mathrm{op}}^{1/2} + \widetilde{C_2} \sqrt{\log n}\cdot n^{-(q-2)/2q} \nonumber\\ =& \sqrt{2\lambda_{t+1}} + \widetilde{C_2}\sqrt{\log n}\cdot n^{-(q-2)/2q} \leq \sqrt{2\lambda_t} + \widetilde{C_2}\cdot \log n\cdot n^{-(q-2)/2q}, 
\end{align}
where $C, \widetilde{C_2} > 0$ are universal constants.

Thus, combining (\ref{triangleineqsplit}), (\ref{firstterm}), (\ref{secondterm}), and (\ref{thirdterm}), we get \[\left(\mathbb{E}S_1^2(\mu, \mu_n)\right)^{1/2} \leq (1+\sqrt{2})\cdot\sqrt{\lambda_t}+ \frac{(C_3(q)+\widetilde{C_2})\log n}{n^{(q-2)/2q}}+ C_3(q)\frac{\log n}{n^{1/t}} \leq 
C_1\sqrt{\lambda_t}+\frac{C_2(q)\log n}{n^{(q-2)/2q}}+ C_3(q)\frac{\log n}{n^{1/t}},\] where $C_1> 0$ is a universal constant and $C_2(q), C_3(q) > 0$ depend only on $q$. Since $t \geq 5$ is fixed arbitrarily, taking the infimum over $t \geq 5$ shows that (\ref{prop2statementforS1}) holds. \qedhere
\end{proof}

\subsection{Part II: Proof of (\ref{upperbd})}
\begin{proof}
Fix $\mu \in \mathcal{P}_q$. Let $\lambda_1 \geq \lambda_2 \geq \ldots \geq 0$ be the eigenvalues of $\int_{\mathcal{H}}(x \otimes x)\text{ }d\mu(x)$. Since $\{\lambda_i\}_{i=1}^{\infty}$ is a non-increasing sequence and (\ref{tracebound}) holds, we have $\lambda_t \leq 1/t$. Hence, applying Proposition \ref{prop2}, we have:
\begin{equation*}
    \left(\mathbb{E}S_k^2(\mu, \mu_n)\right)^{1/2} \leq \inf_{t \geq 5} \left\{ \sqrt{k}\left( \frac{C_1}{\sqrt{t}}+\frac{C_2(q)\log n}{n^{(q-2)/2q}}+C_3(q) \frac{\log n}{n^{1/t}} \right) \right\}.
\end{equation*}

Case 1: Say $k < \log n$ so that $\min\{k, \log n\} = k.$
\begin{itemize}
\item Case 1-a: Say $\left\lfloor \frac{2\log n}{3 \log \log n} \right\rfloor \geq 5.$ Taking $t = \left\lfloor \frac{2\log n}{3\log \log n} \right\rfloor$, we obtain

\begin{equation*}
    \left(\mathbb{E}S_k^2(\mu, \mu_n)\right)^{1/2} \leq \sqrt{k}\left(\frac{\widetilde{C_1} \sqrt{3\log \log n}}{\sqrt{2\log n}} + \frac{C_2(q) \log n}{n^{(q-2)/2q}} + \frac{C_3(q) \cdot \log n}{(\log n)^{3/2}} \right) \leq \frac{C_q\sqrt{k\log \log n}}{\sqrt{\log n}}.
\end{equation*}

\item Case 1-b: Say $\left\lfloor \frac{2\log n}{3\log \log n} \right\rfloor <  5.$ Noting that $\frac{2\log n}{3\log \log n} < 5$ and $k \geq 1$, we use (\ref{Sk<Wk}) to obtain
\begin{equation*}
    \left(\mathbb{E}S_k^2(\mu, \mu_n)\right)^{1/2} \leq \sqrt{2} < \sqrt{2}\cdot \sqrt{5} \cdot\frac{\sqrt{3\log \log n}}{\sqrt{2\log n}} \leq \sqrt{15}\cdot \sqrt{k} \cdot \sqrt{\frac{{\log \log n}}{{\log n}}}. 
\end{equation*}
\end{itemize}

Case 2: Say $k \geq \log n$ so that $\min\{k, \log n\} = \log n.$ Noting that $\log \log n \geq \log \log 3$, we use (\ref{Sk<Wk}) to obtain 
\begin{equation*}
    \left(\mathbb{E}S_k^2(\mu, \mu_n)\right)^{1/2} \leq \sqrt{2} \leq \sqrt{2}\cdot \sqrt{\frac{\log\log n}{\log \log 3}} = \sqrt{\frac{2}{\log \log 3}} \cdot \sqrt{\log n} \cdot \sqrt{\frac{{\log \log n}}{\log n}}. 
\end{equation*}

Since $\mu \in \mathcal{P}_q$ is fixed arbitrarily, (\ref{upperbd}) holds. This completes the proof of the upper bound in Theorem \ref{Mainthm}.\qedhere
\end{proof}

\section{Proof of lower bound in Theorem \ref{Mainthm}}\label{lowerboundsec}

The goal of this section is to prove the lower bound in Theorem \ref{Mainthm}, i.e., for given $q>2, k \in \mathbb{N}$, and $n \geq 3$,
\begin{equation*}
  \sup_{\mu\in\mathcal{P}_q}(\mathbb{E}S_k^2(\mu, \mu_n))^{1/2} \geq c\sqrt{\frac{\min\{k, \log n\}}{{\log n}}},  
\end{equation*} where $c > 0$ is a universal constant. It suffices to show that there exists $\mu \in \mathcal{P}_q$ such that $\left(\mathbb{E}S_k^2(\mu, \mu_n)\right)^{1/2} \geq c\sqrt{{\min\{k, \log n\}}/{{\log n}}}$. 

\begin{proof}Let $d = \lceil \log n \rceil$ and $\widetilde{k}= \min\{k, d \}$ so that $1 \leq \widetilde{k} \leq d$. Let $\nu$ be the uniform distribution on the Euclidean unit ball $B_2^d(0,1)$ in $\mathbb{R}^d$, and let $\mu \coloneqq T_{\#}\nu$, with $T$ as defined in Section \ref{coupling}. Then, $\mu \in \mathcal{P}_q$, since \[\int_{\mathcal{H}}\|x\|^q \text{ }d\mu(x) = \int_{\mathbb{R}^d} \|Tx\|^q \text{ }d\nu(x) = \int_{\mathbb{R}^d} \|x\|^q \text{ }d\nu(x) = \int_{B_2^d(0,1)} \|x\|^q \text{ }d\nu(x) \leq 1,\] where the first step follows from a change of variables, the second step follows from the fact that $T$ is an isometry, and the last two steps follow from the definition of $\nu$. We now apply Lemma \ref{rootk}, Proposition \ref{prop}, Remark \ref{SRWAlternateForms}, and Lemma \ref{unif} to obtain the following. 
\begin{eqnarray}\label{4.1}
    \left(\mathbb{E}S_k^2(\mu, \mu_n)\right)^{1/2} = \left(\mathbb{E}S_k^2(T_{\#}\nu, T_{\#}\nu_n)\right)^{1/2}\nonumber &\geq&\left(\mathbb{E}S_{\widetilde{k}}^2(T_{\#}\nu, T_{\#}\nu_n)\right)^{1/2}\nonumber \\ &\geq&\left(\mathbb{E}S_{\widetilde{k}}^2(\nu, \nu_n)\right)^{1/2} \nonumber \\ &\geq& \left(\frac{\widetilde{k}}{d}\text{ }\mathbb{E}W_2^2(\nu, \nu_n)\right)^{1/2}\nonumber\\ &\geq& \frac{1}{2\sqrt{2}} \cdot \sqrt{\frac{\widetilde{k}}{d}} \cdot n^{-1/d}  \nonumber \\ &\geq& \frac{1}{2\sqrt{2}} \cdot \sqrt{\frac{\widetilde{k}}{2 \log n}} \cdot n^{-1/\log n}\nonumber \\ &\geq& c\sqrt{\frac{\min\{k, \log n\}}{\log n}},
\end{eqnarray} where the last two steps follow because $\log n \leq d = \lceil \log n \rceil \leq 2\log n,$ and $c = (4e)^{-1} > 0$ is a universal constant. This completes the proof of the lower bound in Theorem \ref{Mainthm}.\end{proof}

\section{Possible Refinements}\label{furtherresults}
Given spectral decay conditions on $\int_{\mathcal{H}} (x\otimes x) \, d\mu(x)$, Theorem \ref{Mainthm} can be improved. We consider the example of polynomial decay in this paper to demonstrate this. 

\begin{remark}\label{polydecay}
    Under the assumptions of Theorem \ref{Mainthm}, fix $\alpha > 1$ and $r > 0$. Let $\mathcal{P}_{\alpha, q, r}$ be the subset of $\mathcal{P}_q$ such that for each $\mu \in \mathcal{P}_{\alpha, q,r}$, the eigenvalues $\lambda_1 \geq \lambda_2 \geq \ldots \geq 0$ of $\int_{\mathcal{H}}(x \otimes x) \text{ }d\mu(x)$ satisfy $\lambda_j \leq r j^{-\alpha}$ for each $j$. Then, given $n \geq 3$, we have  
    \[\frac{c_{\alpha, q, r}\sqrt{\min\{k, \log n\}} }{(\log n)^{\alpha /2}} \leq \underset{\mu\in\mathcal{P}_{\alpha, q,r}}{\sup}\left(\mathbb{E}S_k^2(\mu, \mu_n)\right)^{1/2} \leq C_{\alpha, q,r}\sqrt{k} \cdot \left(\frac{\log \log n}{\log n}\right)^{\alpha / 2},\] where both $C_{\alpha, q,r},\, c_{\alpha, q, r}>0$ depend only on $\alpha, q, r$.
\end{remark}
\begin{proof}
\textbf{Upper Bound}: Fix $\mu \in \mathcal{P}_{\alpha, q,r}$. Let $\lambda_1 \geq \lambda_2 \geq \ldots \geq 0$ be the eigenvalues of $\int_{\mathcal{H}}(x \otimes x) \text{ }d\mu(x)$. Then, $\lambda_t \leq rt^{-\alpha}$.

If $\left\lfloor \frac{2 \log n}{(2+\alpha)\log \log n} \right\rfloor \geq 5$, by Proposition \ref{prop2}, we have
\begin{eqnarray}
  \left(\mathbb{E}S_k^2(\mu, \mu_n)\right)^{1/2}&\leq& \inf_{t \geq 5}\left\{  \sqrt{k}\left(C_1\cdot \sqrt{r}t^{-\alpha/2}  +\frac{C_2(q)\log n}{n^{(q-2)/2q}}+C_3(q) \frac{\log n}{n^{1/t}} \right)\right\}\nonumber \\ &\leq& \sqrt{k}\left( \widetilde{C_1}\sqrt{r}\left(\frac{(2+\alpha)\log \log n}{2\log n}\right)^{\alpha/2} + \frac{C_2(q)\log n}{n^{(q-2)/2q}}+ \frac{C_3(q)}{(\log n)^{\alpha/2}} \right)\nonumber\\
  &\leq& C_{\alpha, q, r}\sqrt{k}\left(\frac{\log \log n}{\log n}\right)^{\alpha/2},
\end{eqnarray} where $C_{\alpha, q, r}$ depends only on $\alpha, q, r$.

If instead, $\left\lfloor \frac{2 \log n}{(2+\alpha)\log \log n} \right\rfloor < 5$, noting that $\frac{2 \log n}{(2+\alpha)\log \log n} < 5$, $k \geq 1,$ and $\mathcal{P}_{\alpha, q, r}\subset \mathcal{P}_q$, we apply (\ref{Sk<Wk}) to obtain
\begin{equation*}
    \left(\mathbb{E}S_k^2(\mu, \mu_n)\right)^{1/2} \leq \sqrt{2} < \sqrt{2}\cdot 5^{\alpha/2} \cdot\left(\frac{(2+\alpha)\log \log n}{{2\log n}}\right)^{\alpha/2} \leq \left(\frac{5(2+\alpha)}{2^{1-\frac{1}{\alpha}}}\right)^{\alpha/2}\cdot \sqrt{k}\left(\frac{ \log \log n}{\log n}\right)^{\alpha/2}. 
\end{equation*}
Since $\mu \in \mathcal{P}_{\alpha, q, r}$ is fixed arbitrarily, the proof of the upper bound is complete.  

\textbf{Lower Bound}: Let $d = \lceil \log n \rceil$ and $\widetilde{k}=\min\{k, d\}$ so that $1 \leq \widetilde{k}\leq d$. Let $\gamma \coloneq N(0,1)$ be the standard Gaussian measure on $\mathbb{R},$ and let
\[m_q \coloneq \int_{\mathbb{R}} |z|^q\, d\gamma(z), \quad M \coloneq \min\left\{r, \left(m_q^{2/q}\, \cdot \left(\summ{1}{\infty}\, i^{-\alpha}\right)\right)^{-1}\right\}.\]
Define $\nu \coloneqq N_d(0, \Sigma)$, the centred Gaussian measure on $\mathbb{R}^d$ with covariance matrix $\Sigma = \mathrm{diag}(\sigma_1^2, \ldots, \sigma_d^2)$, where $\sigma_i^2 \coloneq Mi^{-\alpha}, \quad i\in\{1, \ldots, d\}.$ For any $x \in \mathbb{R}^d$, let $x_i$ be the $i$-th coordinate of $x$ in the standard basis. Then, 
\[\left(\int_{\mathcal{H}}\|x\|^q \, d(T_{\#}\nu)(x)\right)^{2/q} = \left(\int_{\mathbb{R}^d} \|x\|^{2\cdot(q/2)}\, d\nu(x)\right)^{2/q} \leq \summ{1}{d}\left(\int_{\mathbb{R}^d}|x_i|^{2\cdot (q/2)}\, d\nu(x)\right)^{2/q} = m_q^{2/q}\cdot \summ{1}{d}\, \sigma_i^2 \leq 1,\] where the first step follows from a change of variables and the fact that $T$ is an isometry, the second step follows from Minkowski's inequality in $L^{q/2}(\nu)$, the third step follows because the $i$-th coordinate marginal of $\nu$ is $N(0, \sigma_i^2)$ and the last step follows from the definitions of $\sigma_i^2$ and $M$. Also, since $T$ is a linear isometry, $\int_{\mathcal{H}}(x\otimes x)\, d(T_{\#}\nu)(x) = T \Sigma T^*$, has eigenvalues, in non-increasing order, $\sigma_1^2 \geq \ldots \geq \sigma_d^2$, followed by zeroes. Since for each $j$, $\sigma_j^2 = Mj^{-\alpha} \leq rj^{-\alpha}$, we have $T_{\#}\nu \in \mathcal{P}_{\alpha, q, r}$. 

Let $\mathcal{L}^d$ denote the Lebesgue measure on $\mathbb{R}^d$. As in Lemma \ref{unif}, we show that for any $x_1, \ldots, x_n \in \mathbb{R}^d,$ if $A \coloneq \overset{n}{\underset{i=1}{\cup}}B_2^d(x_i, \epsilon)$ with $\epsilon \coloneq \sqrt{2 \pi} (2n)^{-1/d} \left(\mathcal{L}^d(B_2^d\left(0,1)\right)\right)^{-1/d}\left(\det \Sigma\right)^{1/2d},$ then
\begin{eqnarray*}W_2^2\left(\nu, \frac{1}{n}\summ{1}{n}\, \delta_{x_i}\right) \geq \epsilon^2\cdot \nu\left(A^c\right) &\geq& \epsilon^2\cdot \left(1 - \summ{1}{n}\,\nu\left(B_2^d(x_i, \epsilon)\right)\right) \nonumber \\
&\geq& \epsilon^2 \cdot \left(1 - n\, (2\pi)^{-d/2}\left(\det \Sigma\right)^{-1/2}\epsilon^d \left(\mathcal{L}^d\left(B_2^d(0,1)\right)\right)\right) \nonumber = \frac{\epsilon^2}{2}\\
&=& \frac{1}{2}(2\pi)(\det\Sigma)^{1/d}\cdot (2n)^{-2/d}\cdot \left(\frac{\pi^{d/2}}{\Gamma\left(\frac{d}{2}+1\right)}\right)^{-2/d} \nonumber \\
&\geq& \frac{M}{8e}\cdot \frac{n^{-2/d}}{d^{\alpha - 1}},
\end{eqnarray*}
where the third step follows because for any measurable $K \subset \mathbb{R}^d,$ \[\nu(K) = \displaystyle\int_{K}\frac{\exp({-\frac{1}{2}\langle\Sigma^{-1}x, x\rangle})}{(2\pi)^{d/2}(\det\Sigma)^{1/2}} \,\, d\mathcal{L}^d(x) \leq \frac{\mathcal{L}^d(K)}{(2\pi)^{d/2}(\det \Sigma)^{1/2}},\] and the final step follows from the following facts: 
\[\mathcal{L}^d(B_2^d(0,1)) =\frac{\pi^{d/2}}{\Gamma \left(\frac{d}{2}+1\right)} \leq \frac{(\pi \cdot 2e)^{d/2}}{{d}^{d/2}},\,\, \quad \,\,\, \det \Sigma = \sigma_1^2 \cdot \sigma_2^2 \cdot \ldots \cdot \sigma_d^2 = \frac{M^d}{(d!)^{\alpha}} \geq \frac{M^d}{(d^{d})^{\alpha}}.\] By randomizing $x_1, \ldots, x_n$ according to $\nu$, we obtain $\mathbb{E}W_2^2(\nu, \nu_n) \geq \frac{M}{8e} \cdot \frac{n^{-2/d}}{d^{\alpha-1}}.$ Using this bound after applying Lemma \ref{rootk}, Proposition \ref{prop}, and Remark \ref{SRWAlternateForms}, as we did for (\ref{4.1}), we get
\begin{eqnarray*}
    \left(\mathbb{E}S_k^2(T_{\#}\nu, T_{\#}\nu_n)\right)^{1/2} \geq \left(\mathbb{E}S_{\widetilde{k}}^2(T_{\#}\nu, T_{\#}\nu_n)\right)^{1/2}  &\geq& \left(\mathbb{E}S_{\widetilde{k}}^2(\nu, \nu_n)\right)^{1/2} \nonumber\\ &\geq& \left(\frac{\widetilde{k}}{d}\mathbb{E}W_2^2(\nu, \nu_n) \right)^{1/2} \nonumber\\&\geq& \frac{c_{\alpha,q,r}\sqrt{\min\{k, \log n\}}}{(\log n)^{\alpha/2}},
\end{eqnarray*} where $c_{\alpha, q, r} = \sqrt{\frac{M}{2^{\alpha}\cdot 8e^3}} > 0$ depends only on $\alpha, q, r$. Since $T_{\#}\nu \in \mathcal{P}_{\alpha, q, r},$ the lower bound holds. 
\end{proof}

\noindent{\bf Acknowledgement:} The author is grateful to March Boedihardjo for his guidance. The author also thanks the anonymous referees for their valuable comments, which helped improve this work.


\begin{thebibliography}{99}
\bibitem{March} M.T.~Boedihardjo, Sharp bounds for max-sliced Wasserstein distances, Foundations of Computational Mathematics (2025)
\bibitem{Chewi} S.~Chewi, J.~Niles-Weed, P.~Rigollet, Statistical Optimal Transport, École d'Été de Probabilités de Saint-Flour XLIX – 2019, Lecture Notes in Mathematics, Springer Cham, 2025, https://doi.org/10.1007/978-3-031-85160-5
\bibitem{Dhouib} S.~Dhouib, I.~Redko, T.~Kerdoncuff, R.~Emonet, M.~Sebban, A Swiss Army Knife for Minimax Optimal Transport, Proceedings of the 37th International Conference on Machine Learning, PMLR 119:2504-2513, 2020.
\bibitem{Dudley} R.M.~Dudley, The speed of mean Glivenko-Cantelli convergence, Annals of Mathematical Statistics 40-50 (1969)
\bibitem{Fournier} N.~Fournier and A.~Guillin, On the rate of convergence in Wasserstein distance of the empirical measure, Probability Theory and Related Fields 162.3-4 (2015): 707-738.
\bibitem{Hilbert} N.~Henze, Random Elements in Separable Hilbert Spaces, Asymptotic Stochastics: An Introduction with a View towards Statistics, Springer Berlin Heidelberg, 2024. 343-386.
\bibitem{Jiang} B.~Jiang, Y.~Liu, A Riemannian Exponential Augmented Lagrangian Method for Computing the Projection Robust Wasserstein Distance, Advances in Neural Information Processing Systems 36 (NeurIPS 2023), New Orleans, U.S.A
\bibitem{Lei} J.~Lei, Convergence and concentration of empirical measures under Wasserstein distance in unbounded functional spaces, Bernoulli 26(1), 2020, 767-798.
\bibitem{Lin2020} T.~Lin, C.~Fan, N.~Ho, M.~Cuturi, M.~I.~Jordan, Projection Robust Wasserstein Distance and Riemannian Optimization, In the 34th Conference on Neural Information Processing Systems (NeurIPS 2020), Vancouver, Canada  
\bibitem{Lin} T.~Lin, Z.~Zheng, E.~Chen, M.~Cuturi, M.~I.~Jordan, On projection robust optimal transport: Sample complexity and model misspecification, In International Conference on Artificial Intelligence and Statistics PMLR 2021.
\bibitem{Nietert} S.~Nietert, Z.~Goldfeld, R.~Sadhu and K.~Kato, Statistical, robustness, and computational guarantees for sliced Wasserstein distances, Advances in Neural Information Processing Systems, 35, 28179-28193 (2022).
\bibitem{NR} J.~Niles-Weed and P.~Rigollet, Estimation of Wasserstein distances in the spiked transport model, Bernoulli 28.4 (2022): 2663-2688.
\bibitem{Panaretos} V.~M.~Panaretos, Y.~Zemel, An Invitation to Statistics in Wasserstein Space, SpringerBriefs in Probability and Mathematical Statistics, 2020 
\bibitem{Paty} F.-P.~Paty and M.~Cuturi, Subspace robust Wasserstein distances, In ICML, pages 5072-5081, 2019.
\bibitem{Vershynin2} R.~Vershynin, Introduction to the non-asymptotic analysis of random matrices, arXiv preprint arXiv:1011.3027v7.
\end{thebibliography}
\end{document}